\documentclass[a4paper,reqno]{amsart}

\usepackage{amssymb}
\usepackage{amstext}
\usepackage{amsmath}
\usepackage{amscd}
\usepackage{amsthm}
\usepackage{amsfonts}
\usepackage{enumerate}
\usepackage{graphicx}
\usepackage{latexsym}
\usepackage{mathrsfs}
\usepackage{mathtools}
\usepackage[all]{xy}
\xyoption{all}

\usepackage{pstricks}
\usepackage{lscape}
\usepackage{comment}

\newtheorem{theorem}{Theorem}[section]

\newtheorem{lemma}[theorem]{Lemma}
\newtheorem{proposition}[theorem]{Proposition}
\newtheorem{definition-proposition}[theorem]{Definition-Proposition}

\newtheorem{question}[theorem]{Question}

\theoremstyle{definition}
\newtheorem{definition}[theorem]{Definition}

\newtheorem{example}[theorem]{Example}

\newcommand{\Ext}{\operatorname{Ext}\nolimits}

\newcommand{\Hom}{\operatorname{Hom}\nolimits}

\renewcommand{\mod}{\mathsf{mod}\hspace{.01in}}

\newcommand{\add}{\mathsf{add}\hspace{.01in}}
\newcommand{\Fac}{\mathsf{Fac}\hspace{.01in}}

\newcommand{\RHom}{\mathbf{R}\strut\kern-.2em\operatorname{Hom}\nolimits}

\newcommand{\Sub}{\mathsf{Sub}\hspace{.01in}}

\numberwithin{equation}{section}

\usepackage{paralist}

\def\Ker{\mathop{\rm Ker}\nolimits}

\def\pd{\mathop{\rm pd}\nolimits}

\begin{document}
\title{Self-orthogonal $\tau$-tilting modules and tilting modules}
\thanks{The author is supported by NSFC(No.11671174), the Project Funded by the Priority
Academic Program Development of Jiangsu Higher Education Institutions and the Starting Funding of Jiangsu Normal University }


\author{Xiaojin Zhang}
\address{X. Zhang: { School of Mathematics and Statistics, Jiangsu Normal University, Xuzhou, 221116, P. R. China}} \email{xjzhang@jsnu.edu.cn, xjzhangmaths@163.com}

\maketitle
\begin{abstract}
Let $\Lambda $ be an artin algebra and $T$ a $\tau$-tilting $\Lambda$-module. We prove that $T$ is a tilting module if and only if ${\rm Ext}_{\Lambda}^{i}(T,\Fac T)=0$ for all $i\geq 1$, where $\Fac T$ is the full subcategory consisting of modules generated by $T$. Consequently, a $\tau$-tilting module $T$ of finite projective dimension is a tilting module if and only if ${\rm Ext}_{\Lambda}^{i}(T, T)=0$ for all $i\geq 1$. Moreover, we also give an example to show that a support $\tau$-tilting but not $\tau$-tilting module $M$ of finite projective dimension satisfying ${\rm Ext}_{\Lambda}^{i}(M, M)=0$ for all $i\geq1$ need not be a partial tilting module.
\end{abstract}

\section{Introduction}

Tilting modules \cite{HaR} have been central in the representation theory of finite dimensional algebras since 1970s. They are also useful in both  commutative algebra \cite{I} and algebraic geometry \cite{Hu}. By a tilting module, we mean a tilting module of projective dimension at most one. There are two different generalizations of tilting modules. One is the notion of tilting modules of finite projective dimension introduced by Miyashita \cite{Mi} and Happel independently. The other is the notion of $\tau$-tilting modules introduced by Adachi, Iyama and Reiten \cite{AIR} in 2014. So it is interesting to study the intersection between tilting modules of finite projective dimension and $\tau$-tilting modules.

It was shown in \cite[Proposition 2.2(b)]{AIR} that tilting modules are precisely faithful support $\tau$-tilting modules. Moreover, by using the Auslander-Reiten formula, one can show that tilting modules are precisely the $\tau$-tilting modules of projective dimension at most one. The comparison of tilting modules and $\tau$-tilting modules attracted a lot of attention (see \cite{BZ1, BZ2, BZ3, IZ1, IZ2, XZZ, Z}).

In this paper, we compare tilting modules and $\tau$-tilting modules from a new viewpoint. On one hand, a tilting $\Lambda$-module $T$ is always self-orthogonal, that is, ${\rm Ext}_{\Lambda}^{i}(T,T)=0$ for all $i\geq 1$. On the other hand, Xie, Zan and the author \cite{XZZ} studied the homological properties of self-orthogonal $\tau$-tilting modules and proved that self-orthogonal $\tau$-tilting modules admit very similar properties to tilting modules. So it is interesting to study how far self-orthogonal $\tau$-tilting modules are from tilting modules. A natural question is:

\begin{question}\label{0.a} Is a self-orthogonal $\tau$-tilting module a tilting module?
\end{question}

 Recall from \cite{AuS} that a module $M$ is $\tau$-rigid if and only if ${\rm Ext}_{\Lambda}^{1}(M,{\rm Fac}M)=0$, where ${\rm Fac}M$ is the full subcategory of $\mod\Lambda$ consisting of modules generated by $M$. It is natural to consider the modules satisfying ${\rm Ext}_{\Lambda}^{i}(M,{\rm Fac}M)=0$ for all $i\geq 1$. By studying the modules satisfying ${\rm Ext}_{\Lambda}^{i}(M,{\rm Fac}M)=0$ for all $i\geq 1$, we prove the following main result in this paper.
\begin{theorem}\label{0.1} Let $\Lambda$ be an artin algebra and $T\in\mod\Lambda$. Then $T$ is a tilting module if and only if $|T|=|\Lambda|$ and ${\rm Ext}_{\Lambda}^{i}(T,\Fac T)=0$ for all $i\geq1$.
\end{theorem}

 Applying Theorem 1.2 to $\tau$-tilting modules of finite projective dimension, we give a positive answer to Question \ref{0.a} as follows:
 \begin{theorem}\label{0.2} (Theorem \ref{2.8}) Let $\Lambda$ be an artin algebra and $T\in\mod\Lambda$ a $\tau$-tilting module of finite projective dimension. Then $T$ is a tilting module if and only if ${\rm Ext}_{\Lambda}^{i}(T,T)=0$ for all $i\geq1$.
 \end{theorem}

In particular, if $\Lambda$ is an artin algebra of finite global dimension, then every self-orthogonal $\tau$-tilting module is a tilting module.

The paper is organized as follows: In Section 2, we recall some preliminaries. In Section 3, we show Theorems \ref{0.1} and \ref{0.2}. Moreover, we give an example to show that a self-orthogonal support $\tau$-tilting but not $\tau$-tilting module of finite projective dimension need not be a partial tilting module.

Throughout this paper, all algebras are artin $R$-algebras, where $R$ is a commutative artin ring.
For an algebra $\Lambda$, we use $\mathrm{mod}\Lambda$ to denote the category of finitely
 generated right $\Lambda$-modules. Denote by $\mathbb{D}=\mathrm{Hom}_{R} (-,E(R/J(R)))$ the standard duality, where $J(R)$ is the Jocobson radical of $R$ and $E(R/J(R))$ is the injective envelope of $R/J(R)$. Denote by $\tau$ the Auslander-Reiten translation functor.

\section{Preliminaries}

In this section, we recall some basic properties of tilting modules and $\tau$-tilting modules.

For an algebra $\Lambda$ and a module $M\in \mod\Lambda$, denote by $|M|$ the number of non-isomorphic indecomposable direct summands of $M$ and denote by $\pd_{\Lambda}M$ the projective dimension of $M$.

\begin{definition}\label{2.1} Let $T\in\mod\Lambda$. $T$ is called a {\it tilting} module if the following are satisfied:
\begin{enumerate}[\rm(1)]
\item $\pd_{\Lambda}T\leq 1$,
\item $\Ext_{\Lambda}^{i}(T,T)=0$ for all $i\geq 1$,
\item There is an exact sequence $0\rightarrow \Lambda\rightarrow T_0\rightarrow T_1\rightarrow0$ with $T_i\in \add T$.
\end{enumerate}
We should remark that (3) is equivalent to $|T|=|\Lambda|$. Moreover, $T$ is called a {\it partial tilting module} if $T$ satisfies (1) and (2).
\end{definition}

We also need the following definitions in \cite{S} and \cite{AIR}.

\begin{definition}\label{2.2}
 \begin{enumerate}[\rm(1)]
\item We call $M\in \mod\Lambda$ {\it $\tau$-rigid} if $\mathrm{Hom}_{\Lambda}(M,\tau M)=0$.

 \item $M \in
\mathrm{mod}\Lambda$ is called {\it $\tau$-tilting} if $M$ is $\tau$-rigid
and $|M|=|\Lambda|$.
\item $M$ is support $\tau$-tilting if $T$ is $\tau$-tilting over the algebra $\Lambda/(e)$, where $e$ is an idempotent of $\Lambda$.
\end{enumerate}
\end{definition}

The following lemma from \cite[Lemma 2.6]{AIR} is essential in this paper.

\begin{lemma}\label{2.b}Let $0\rightarrow Y\rightarrow T_0\stackrel{f}{\rightarrow} X$ be an exact sequence in $\mod\Lambda$, where $T$ is $\tau$-rigid, and $f:T_0\rightarrow X$ is a minimal right $\add T$-approximation of $X$. Then $\Hom_{\Lambda}(Y,\tau T)=0$.
\end{lemma}

We also need the following theorem from \cite[Theorem 2.12]{AIR}.
\begin{theorem}\label{2.c} Let $T\in\mod\Lambda$ be a $\tau$-rigid module. Then the following are equivalent.
\begin{enumerate}[\rm(1)]
\item $T$ is a $\tau$-tilting module.
\item $(\Fac T,\Sub \tau T)$ is a torsion pair.
\end{enumerate}
\end{theorem}

\begin{proof}$(1)\Rightarrow(2)$ Since $T$ is $\tau$-rigid, one gets a torsion pair $(^{\bot_0}\tau T,\Sub\tau T)$ by \cite[Lemma 2.11]{AIR}, where $^{\bot_0}\tau T$ is the subcategory of $\mod\Lambda$ consisting of modules $M$ satisfying $\Hom_\Lambda(M,\tau T)=0$ . Then by \cite[Theorem 2.12]{AIR}, one gets $^{\bot_0}\tau T=\Fac T$, the assertion holds.

$(2)\Rightarrow(1)$ Since $T$ is $\tau$-rigid, one gets a torsion pair $(^{\bot_0}\tau T,\Sub\tau T)$ by \cite[Lemma 2.11]{AIR}. The fact $(\Fac T,\Sub \tau T)$ is a torsion pair implies that $^{\bot_0}\tau T=\Fac T$. Then by \cite[Theorem 2.12]{AIR}, $T$ is a $\tau$-tilting module.
\end{proof}

The following proposition on $\tau$-tilting modules is important.

\begin{proposition}\label{2.3} Let $T\in\mod \Lambda$ be a $\tau$-tilting module. For any $M\in\Fac T$, there is an exact sequence $\cdots\rightarrow T_1\stackrel{f_1}{\rightarrow}T_0\stackrel{f_0}{\rightarrow}M\rightarrow 0$ with $T_i\in \add T$ and $\Ker f_i\in \Fac T$.
\end{proposition}

\begin{proof}  Since $M\in\Fac T$, then one gets an epimorphism $g:T'\rightarrow M$ with $T'\in \add T$. Let $T_0\stackrel{f_0}{\rightarrow}M$ be a minimal right $\add T$-approximation of $M$. Then one gets that $f_0$ is an epimorphism. Since $T$ is $\tau$-tilting, by Lemma \ref{2.b} we have $\Hom_{\Lambda}(\Ker f_0, \tau T)=0$. Using the fact $T$ is a $\tau$-tilting module again, we have a torsion pair $(\Fac T, \Sub \tau T)$ by Theorem \ref{2.c}. Then $\Ker f_0\in\Fac T$.  Using a similar argument, one gets the assertion.
\end{proof}

The following theorem is known as the Auslander-Reiten formula \cite[Chapter IV, Theroem 2.13]{AsSS}.

\begin{theorem}\label{2.4} For any $M,N\in \mod\Lambda$,
$\Ext_{\Lambda}^1(M,N)\simeq \mathbb{D}\overline{\Hom}_{\Lambda}(N,\tau M)$ holds.
\end{theorem}

\section{Main results}

In this section, we show that a self-orthogonal $\tau$-tilting module of finite projective dimension is a tilting module. In addition, we give an example to show that a self-orthogonal support $\tau$-tilting but not $\tau$-tilting module of finite projective dimension need not be a partial tilting module.

 The following property of $\tau$-tilting modules which is a generalization of that of tilting modules is quite essential in this paper.

\begin{proposition}\label{2.5} Let $T\in\mod\Lambda$ be a $\tau$-tilting module and let $M\in \mod\Lambda$ satisfy $\Ext_{\Lambda}^1(T,M)=0$. Then $\Hom_{\Lambda}(T,M)=0$ holds if and only if $M=0$.
\end{proposition}
\begin{proof} $\Rightarrow$ On the contrary, suppose that $M\not=0$. Then the injective envelope $I^0(M)$ is not $0$. Since $T$ is a $\tau$-tilting module, one gets a torsion pair $({\rm Fac}T, {\rm Sub}\tau T)$ by Theorem \ref{2.c}. Then $\Hom_{\Lambda}(T,M)=0$ implies that $M\in {\rm Sub}\tau T$. That is, $M\rightarrowtail (\tau T)^m$ for some integer $m\geq 1$. Since $\Ext_{\Lambda}^{1}(T,M)=0$, by Theorem \ref{2.4}, one gets $\overline{\Hom}_{\Lambda}(M,\tau T)=0$ which implies that the monomorphism $f:M\rightarrowtail (\tau T)^m$ factors through an injective module, and hence factors through the injective envelope $h:M\hookrightarrow I^0(M)$, that is, $f=gh$ for some $g: I^0(M)\rightarrow (\tau T)^m$. Since $h$ is an essential monomorphism, one gets that $g$ is a monomorphism, and hence $I^0(M)$ is a direct summand of $\tau T$. This is a contradiction.

$\Leftarrow$ This is obvious.
\end{proof}

For a module $T\in\mod\Lambda$, denote by $T^{\bot}=\{M\in\mod\Lambda|\Ext_\Lambda^i(T,M)=0,\forall\  i\geq1\}$. Now we show the relations between self-orthogonal $\tau$-tilting modules and tilting modules.

\begin{theorem}\label{2.6} Let $T\in\mod\Lambda$ be a $\tau$-tilting module. Then $T$ is a tilting module if and only if $\Ext_{\Lambda}^{i}(T,\Fac T)=0$ for all $i\geq 1$.
\end{theorem}
\begin{proof}$\Rightarrow$ This is obvious since $\Fac T=T^{\bot}$ whenever $T$ is a tilting module.

$\Leftarrow$ Recall that $T$ is tilting if and only if $\Fac T=T^{\bot}$. So we only need to show that $T^{\bot}\subseteq \Fac T$, that is, for any $0\not=M\in T^{\bot}$, we have $M\in \Fac T$.

We use induction on the length $l(M)$ of $M$.

If $l(M)=1$, then $M$ is simple. $M\in T^{\bot}$ implies that $\Hom_{\Lambda}(T,M)\not=0$ by Proposition \ref{2.5}. Then one gets that $M\in\Fac T$.

Assume that the assertion holds for $l(M)\leq k$. Now we show the case $l(M)=k+1$. Since $T$ is $\tau$-tilting, by Theorem \ref{2.c} we get the following canonical exact sequence given by the torsion pair ($\Fac T, \Sub\tau T$):
$$0\rightarrow T(M)\rightarrow M\rightarrow F(M)\rightarrow0\ \ \ \ \ \ (*)$$
with $T(M)\in \Fac T$ and $F(M)\in \Sub\tau T$. By the assumption $\Fac T\subseteq T^{\bot}$, we get $\Ext_{\Lambda}^{i}(T,T(M))=0$ for all $i\geq 1$.

We claim that $T(M)\neq 0$. If $T(M)=0$, then $M\simeq F(M)$. So $M\in T^{\bot}$ implies $F(M)\in T^{\bot}$. Since $\Hom_{\Lambda}(T,F(M))=0$, by Proposition \ref{2.5} one gets that $F(M)=0$ which implies $M=0$. This is a contradiction.

 Applying the functor $\Hom_{\Lambda}(T,-)$ to the exact sequence $(*)$ above, one gets that $\Ext_{\Lambda}^{i}(T,F(M))=0$ for all $i\geq 1$. Since $l(F(M))<l(M)=k+1$, we get $F(M)\in \Fac T$ by induction hypothesis. Then $F(M)\in\Sub\tau T$ implies that $F(M)=0$ and hence $M=T(M)$.
\end{proof}

Now we show the proof of Theorem \ref{0.1}.

\begin{proof} By using a result of Auslander and Smal$\o$, a module $T\in\mod\Lambda$ satisfying $\Ext_{\Lambda}^{i}(T,\Fac T)=0$ for all $i\geq 1$ is always $\tau$-rigid. Then by Theorem \ref{2.6}, one gets that $T$ is tilting if and only if $|T|=|\Lambda|$ and $\Ext_{\Lambda}^{i}(T,\Fac T)=0$ for all $i\geq 1$.
\end{proof}

To show the main result on $\tau$-tilting modules of finite projective dimension, we need the following lemma.

\begin{lemma}\label{2.7} Let $T\in\mod\Lambda$ be a $\tau$-tilting module with $\pd_{\Lambda}T=d<\infty$. If $\Ext_{\Lambda}^{i}(T,T)=0$ holds for all $i\geq 1$, then $\Fac T\subseteq T^{\bot}$.
\end{lemma}

\begin{proof} For any $M\in \Fac T$, by Proposition \ref{2.3}, we get a long exact sequence

$$\cdots\rightarrow T_1\stackrel{f_1}{\rightarrow} T_0\stackrel{f_0}{\rightarrow} M\rightarrow0\ \ \ \ \ \ \ (1)$$ with $T_i\in \add T$ and $\Ker f_i\in \Fac T$. Consider the sequence: $$0\rightarrow \Ker f_{d-1}\rightarrow T_{d-1}\stackrel{f_{d-1}}{\rightarrow}T_{d-1}\stackrel{f_{d-2}}{\rightarrow}\cdots\rightarrow T_0\stackrel{f_{0}}{\rightarrow}M\rightarrow 0\ \ (2)$$ Since all $T_i\in T^{\bot}$, we have by dimension shifting that
$$\Ext^i_{\Lambda}(T, M)\simeq\Ext^{i+d}_{\Lambda}(T, \Ker f_{d-1}) =0$$ for all $i\geq1$.  We are done.
\end{proof}

Now we are in a position to show the following result.

\begin{theorem}\label{2.8} Let $\Lambda$ be an artin algebra and $T\in\mod\Lambda$ a $\tau$-tilting module of finite projective dimension. Then $T$ is a tilting module if and only if ${\rm Ext}_{\Lambda}^{i}(T,T)=0$ for all $i\geq1$.
\end{theorem}

\begin{proof}$\Rightarrow$  This is obvious.

$\Leftarrow$ By Lemma \ref{2.7}, $\Fac T\subseteq T^{\bot}$. Then the assertion follows from Theorem \ref{2.6}.
\end{proof}

We should mention that the statement of Theorem \ref{0.1} may fail over rings which are not artin algebras. The following example was due to
Angeleri H\"{u}gel and Hrbek \cite[Example 5.4]{AnH}.

\begin{example}\label{2.14} Let $R$ be a commutative local ring with a non-zero idempotent maximal ideal $m$. Let $T=R/m$ be the simple $R$-module.
Then $T$ satisfies that $\Ext^{i}_{R}(T,\Fac T)=0$ for all $i\geq 1$ and has the right number of summands since $R$ is local. But $T$ is not a tilting module.
\end{example}

We give the following example to show that a self-orthogonal support $\tau$-tilting but not $\tau$-tilting module need not be a partial tilting module.

\begin{example}\label{2.13} Let $\Lambda$ be the finite dimensional algebra given by the quiver $Q: 1\stackrel{a_1}{\rightarrow} 2\stackrel{a_2}{\rightarrow} 3$ with the relation $a_1a_2=0$.
Then $S(1)=1$ is an injective support $\tau$-tilting module with projective dimension $2$ and hence not a partial tilting module.
\end{example}

Recall that Tachikawa \cite{T} conjectured that over a self-injective algebra $\Lambda$, every self-orthogonal module $M\in\mod\Lambda$ is projective. When $\Lambda$ is self-injective, Tachikawa's conjecture coincides with the Generalized Nakayama Conjecture of Auslander and Reiten \cite{AuR} which states that a self-orthogonal generator is projective. We end the paper with the following question.

\begin{question}\label{2.10} Let $\Lambda$ be a self-injective algebra and $T \in\mod\Lambda$. If $T$ is a self-orthogonal $\tau$-tilting module, is $T$ a tilting module?
\end{question}

A positive answer to the question will be a partial answer to Tachikawa's conjecture as well as a partial answer to the Generalized Nakayama Conjecture. While a negative answer will give a counter-example to both conjectures.

\vspace{0.2cm}

{\bf Acknowledgement} The author would like to thank Xiao-Wu Chen, Zhaoyong Huang, Osamu Iyama, Yu Zhou and Bin Zhu for useful suggestions to improve the paper. He also wants to thank the referee for his/her useful suggestions to improve this paper.

\end{document}